\documentclass[11pt,reqno]{amsart} 
\usepackage{amssymb,amsthm,amsmath,epsfig,latexsym}
\usepackage{graphicx}
\usepackage{calc,times}

\usepackage{geometry} 
\begin{document}

\newcommand{\mmbox}[1]{\mbox{${#1}$}}
\newcommand{\proj}[1]{\mmbox{{\mathbb P}^{#1}}}
\newcommand{\Cr}{C^r(\Delta)}
\newcommand{\CR}{C^r(\hat\Delta)}
\newcommand{\affine}[1]{\mmbox{{\mathbb A}^{#1}}}
\newcommand{\Ann}[1]{\mmbox{{\rm Ann}({#1})}}
\newcommand{\caps}[3]{\mmbox{{#1}_{#2} \cap \ldots \cap {#1}_{#3}}}
\newcommand{\N}{{\mathbb N}}
\newcommand{\Z}{{\mathbb Z}}
\newcommand{\R}{{\mathbb R}}
\newcommand{\Tor}{\mathop{\rm Tor}\nolimits}
\newcommand{\Ext}{\mathop{\rm Ext}\nolimits}
\newcommand{\Hom}{\mathop{\rm Hom}\nolimits}
\newcommand{\im}{\mathop{\rm Im}\nolimits}
\newcommand{\rank}{\mathop{\rm rank}\nolimits}
\newcommand{\supp}{\mathop{\rm supp}\nolimits}
\newcommand{\arrow}[1]{\stackrel{#1}{\longrightarrow}}
\newcommand{\CB}{Cayley-Bacharach}
\newcommand{\coker}{\mathop{\rm coker}\nolimits}
\newcommand{\m}{{\frak m}}
\newcommand{\fitt}{{\rm Fitt}}
\newcommand{\dep}{{\rm depth}}
\newcommand{\C}{{\mathcal{C}}}
\newcommand{\D}{{\mathcal{D}}}
\newcommand{\K}{\mathbb K}
\newcommand{\OT}{R}  
\newcommand{\spn}{{\rm Span }}  
\newcommand{\M}{\mathsf M}
\newcommand{\Ima}{{\rm Im}\,}
\newcommand{\hi}{{\rm ht}}


\makeatletter
\renewcommand*\env@matrix[1][*\c@MaxMatrixCols c]{%
  \hskip -\arraycolsep
  \let\@ifnextchar\new@ifnextchar
  \array{#1}}
\makeatother


\sloppy
\newtheorem{defn0}{Definition}[section]
\newtheorem{prop0}[defn0]{Proposition}
\newtheorem{quest0}[defn0]{Question}
\newtheorem{thm0}[defn0]{Theorem}
\newtheorem{lem0}[defn0]{Lemma}
\newtheorem{corollary0}[defn0]{Corollary}
\newtheorem{example0}[defn0]{Example}
\newtheorem{remark0}[defn0]{Remark}
\newtheorem{prob0}[defn0]{Problem}

\newenvironment{defn}{\begin{defn0}}{\end{defn0}}
\newenvironment{prop}{\begin{prop0}}{\end{prop0}}
\newenvironment{quest}{\begin{quest0}}{\end{quest0}}
\newenvironment{thm}{\begin{thm0}}{\end{thm0}}
\newenvironment{lem}{\begin{lem0}}{\end{lem0}}
\newenvironment{cor}{\begin{corollary0}}{\end{corollary0}}
\newenvironment{ex}{\begin{example0}\rm}{\end{example0}}
\newenvironment{rem}{\begin{remark0}\rm}{\end{remark0}}
\newenvironment{prob}{\begin{prob0}\rm}{\end{prob0}}

\newcommand{\defref}[1]{Definition~\ref{#1}}
\newcommand{\propref}[1]{Proposition~\ref{#1}}
\newcommand{\thmref}[1]{Theorem~\ref{#1}}
\newcommand{\lemref}[1]{Lemma~\ref{#1}}
\newcommand{\corref}[1]{Corollary~\ref{#1}}
\newcommand{\exref}[1]{Example~\ref{#1}}
\newcommand{\secref}[1]{Section~\ref{#1}}
\newcommand{\remref}[1]{Remark~\ref{#1}}
\newcommand{\questref}[1]{Question~\ref{#1}}
\newcommand{\probref}[1]{Problem~\ref{#1}}

\newcommand{\std}{Gr\"{o}bner}
\newcommand{\jq}{J_{Q}}
\def\Ree#1{{\mathcal R}(#1)}

\numberwithin{equation}{subsection}  


\title{Subspace Arrangements as Generalized Star Configurations}
\author{\c{S}tefan O. Toh\v{a}neanu}

 \keywords{subspace arrangement, star configuration, interpolation, arithmetic rank. \\ \indent Tohaneanu's Address: Department of Mathematics, University of Idaho, Moscow, Idaho 83844-1103, USA, Email: tohaneanu@uidaho.edu, Phone: 208-885-6234, Fax: 208-885-5843.}

\begin{abstract}
In these notes we show that any projective subspace arrangement can be realized as a generalized star configuration variety. This type of interpolation result may be useful in designing linear codes with prescribed codewords of minimum weight, as well as in answering a couple of questions asked by the author in previous work, about the number of equations needed to define a generalized star configuration.
\end{abstract}
\maketitle

\section{Introduction}

Let $\Lambda=(\ell_1,\ldots,\ell_n)$ be a collection of linear forms (some possibly proportional) in $R:=\mathbb K[x_1,\ldots,x_k]$, where $\mathbb K$ is any field, and $k\geq 2$. Suppose that $\langle \ell_1,\ldots,\ell_n\rangle=\langle x_1,\ldots,x_k\rangle=:\frak m$. Let $I_a(\Lambda)\subset R$ be the ideal generated by all $a$-fold products of the linear forms in $\Lambda$, i.e.
\[
I_a(\Lambda)=\langle \ell_{i_1}\cdots\ell_{i_a}|1\leq i_1<\cdots<i_a\leq n \rangle.
\]
The projective scheme with defining ideal $I_a(\Lambda)$ will be called a {\em generalized star configuration scheme (GSCS) of size $a$ and support $\Lambda$}. The variety (subspace arrangement) of $\mathbb P^{k-1}$ with defining ideal $\sqrt{I_a(\Lambda)}$ will be called {\em generalized star configuration variety (GSCV) of size $a$ and support $\Lambda$}, and it will be denoted $\mathcal V_a(\Lambda)$.

Beginning of Section 2 in \cite{To1} gives that for any $a=1,\ldots, n$, one has $$\sqrt{I_a(\Lambda)}=\bigcap_{1\leq i_1<\cdots<i_{n-a+1}\leq n}\langle \ell_{i_1},\ldots,\ell_{i_{n-a+1}}\rangle.$$ So, if any $k$ of the linear forms of $\Lambda$ are linearly independent, then for $n-k+2\leq a\leq n$, $\mathcal V_a(\Lambda)$ is a codimension $(n-a+1)-$ star configuration, according to \cite{GeHaMi}. These already known varieties will be called {\em usual} star configurations.

GSCV's, and more generally GSCS's, are strongly related to coding theory. For example, their dimensions are determined by the generalized Hamming weights of the linear code built on $\Lambda$; for more details and various other properties, see \cite{AnGaTo}. The usual star configurations in fact correspond to Maximum Distance Separable codes.

In the first part of these notes we show that any projective subspace arrangement can be interpolated by GSCV's. This is a multivariate type of interpolation and as any interpolation this can be applied to signal processing, or computer aided design, just to name a few. This technique may be used in designing linear codes with prescribed maximal subcodes of given weight, and more specifically, to design codes with prescribed codewords of minimum weight.

In the second part we study the arithmetic rank of GSCV's. We answer a couple of questions mentioned in \cite{To2}, and we make some comments from our perspective in regard to the arithmetic rank of projective subspace arrangements.

\section{Interpolating with Generalized Star Configuration Varieties}

Let $\K$ be an infinite field. Let $V=V_1\cup\cdots\cup V_m\subset\mathbb P_{\mathbb K}^{k-1}$ be a subspace arrangement of $m$ irreducible components such that $V_1\cap\cdots\cap V_m=\emptyset$. Suppose that the codimension of each irreducible component is $c_i, i=1,\ldots,m$. In the coordinate ring $R:=\mathbb K[x_1,\ldots,x_k]$, the defining ideals $I(V_i)$ are prime ideals minimally generated by $c_i$ linear forms, $$I(V)=I(V_1)\cap\cdots\cap I(V_m),$$ and if $\frak m:=\langle x_1,\ldots,x_k\rangle$ is the irrelevant maximal ideal, we have $I(V_1)+\cdots+I(V_m)=\frak m$. Because of this last condition we are going to say that $V$ is {\em essential}.

\begin{thm}\label{mixdim} Let $V$ be an essential subspace arrangement as above. Then there exist a collection of linear forms $\Lambda=(\ell_1,\ldots,\ell_n)$ where $\ell_i\in R$ generating $\frak m$, and an $a\in\{1,\ldots,n\}$, such that $V=\mathcal V_a(\Lambda)$.
\end{thm}
\begin{proof} For each $i=1,\ldots,m$, let $\Lambda_i$ be a collection of $\displaystyle\aleph:=1+\sum_{i=1}^m(c_i-1)$ linear forms such that each $c_i$ of them generate $I(V_i)$. Since $\mathbb K$ is infinite, $\Lambda_i$'s exist.

Let $\displaystyle\Lambda:=\bigcup_{i=1}^m\Lambda_i$ be the collection of all these $n:=|\Lambda|\leq m\aleph$ linear forms.

Let $a=n-\aleph+1$, hence $\aleph=n-a+1$.

As mentioned in the Introduction, any minimal prime of $I_a(\Lambda)$ is generated (not minimally) by $n-a+1$ linear forms of $\Lambda$, and conversely, any ideal generated by $n-a+1$ linear forms of $\Lambda$ is a prime ideal containing $I_a(\Lambda)$. Therefore for all $i=1,\ldots,m$, $I(V_i)$ is a minimal prime of $I_a(\Lambda)$, hence $$\sqrt{I_a(\Lambda)}\subseteq I(V_1)\cap\cdots\cap I(V_m).$$

Consider some collection of $n-a+1=\aleph$ linear forms from $\Lambda$, and suppose that this collection is obtained by putting together $c_1-\varepsilon_1, \varepsilon_1\geq 0$ linear forms from $\Lambda_1$, $c_2-\varepsilon_2, \varepsilon_2\geq 0$ linear forms from $\Lambda_2$, and so forth, $c_m-\varepsilon_m, \varepsilon_m\geq 0$ linear forms from $\Lambda_m$. Then $$\varepsilon_1+\varepsilon_2+\cdots+\varepsilon_m=\sum_{i=1}^mc_i-\aleph=m-1.$$ Hence at least one of the $\varepsilon_i$'s must be zero.

This is saying that any collection of $n-a+1$ linear forms from $\Lambda$ contains at least $c_{i_0}$ linear forms from the same $\Lambda_{i_0}$. Since these $c_{i_0}$ linear forms will generate $I(V_{i_0})$, we conclude that any minimal prime of $I_a(\Lambda)$ contains one of the $I(V_i)$. Therefore, $$I(V_1)\cap\cdots\cap I(V_m)\subseteq \sqrt{I_a(\Lambda)},$$ which completes the proof.
\end{proof}

\begin{ex}\label{example} Let $V$ be the subspace arrangement with defining ideal $I(V)=\langle x,z,w\rangle\cap\langle x,y\rangle\subset R:=\mathbb C[x,y,z,w]$. We have $m=2$, $c_1=3$, $c_2=2$, hence $\aleph=4$. We can pick $$\Lambda_1=\{x,z,w,x+z+w\}\mbox{ and }\Lambda_2=\{x,y,x+y,x-y\}.$$ Observe that any three of the linear forms in $\Lambda_1$ generate $\langle x,z,w\rangle$, and any two of the linear forms in $\Lambda_2$ generate $\langle x,y\rangle$. With $\Lambda=\{x,z,w,x+z+w,y,x+y,x-y\}$, we have $n=7$ and $a=7-4+1=4$, and conclude that $\mathcal V_4(\Lambda)$ interpolates $V$.
\end{ex}

\subsection{The connection with coding theory.} Let $\mathcal C$ be an $[n,k]$-linear code: $\mathcal C=Im(\phi_G)$, where $\phi_G:\mathbb K^k\rightarrow \mathbb K^n$ is the multiplication by a $k\times n$ matrix $G$ (called {\em generating matrix}), of rank $k$. Let $\mathcal D\subseteq \mathcal C$ be a subcode. The support of $\mathcal D$ is $$Supp(\mathcal D):=\{i :\exists (y_1,\ldots,y_n)\in \mathcal D \mbox{ with }y_i\neq 0\}.$$ Let $m(\mathcal D):=|Supp(\mathcal D)|$ be the cardinality of the support of $\mathcal D$.

Let $V_{\mathcal D}:=\phi_G^{-1}(\mathcal D)$ be the corresponding linear subspace of $\mathbb K^k$, the preimage of $\mathcal D$ under the injective linear map $\phi_G$.

Suppose $\ell_1,\ldots,\ell_n\in R:=\mathbb K[x_1,\ldots,x_k]$ are the linear forms dual to the columns of the matrix $G$: for $i=1,\ldots,n$, the $i-$th column of $G$, $\displaystyle\left[\begin{array}{c}a_{1i}\\a_{2i}\\\vdots\\a_{ki}\end{array}\right]$ has dual linear form $\ell_i:=a_{1i}x_1+a_{2i}x_2+\cdots+a_{ki}x_k$.

If $m(\mathcal D)=s$, then all the elements of $\mathcal D$ have the same $n-s$ components $i_1,\ldots,i_{n-s}\in \{1,\ldots,n\}$ equal to zero. So, $V_{\mathcal D}\subseteq V(\ell_{i_1},\ldots,\ell_{i_{n-s}})$, the common zero locus of these linear forms, and therefore, in terms of defining ideals in $R$, $$\langle \ell_{i_1},\ldots,\ell_{i_{n-s}}\rangle\subseteq I(V_{\mathcal D}).$$

\medskip

Let $V_1,\ldots, V_m$ be the components of an essential subspace arrangement $V$ as above, and let $\Lambda=\{\ell_1,\ldots,\ell_n\}$ be the set of linear forms from the proof of Theorem \ref{mixdim}.

Let $\mathcal C_{\Lambda}$ be the $[n,k]-$linear code with generating matrix $G_{\Lambda}$, whose columns are dual to the linear forms of $\Lambda$. For each $i=1,\ldots,m$, we have $V_i=V(\Lambda_i)$. Let $$\mathcal D_i:=\phi_{G_{\Lambda}}(V_i), i=1,\ldots,m.$$ These are subcodes of $\mathcal C_{\Lambda}$ of support size $$m(\mathcal D_i)\leq |\Lambda|-|\Lambda_i|=n-\aleph=a-1.$$ We have inequality because it may be possible that we have chosen an $\ell\in \Lambda_j\setminus \Lambda_i, j\neq i$, yet $\ell(V_i)=0$.

Let $\mathcal D$ be some nonzero subcode of support size $s\leq a-1$. Then, as we have seen before, $\langle \ell_{i_1},\ldots,\ell_{i_{n-s}}\rangle\subseteq I(V_{\mathcal D})$, for some $\ell_{i_1},\ldots,\ell_{i_{n-s}}\in\Lambda$. Since $s\leq a-1$, then $n-s\geq n-a+1$, and therefore, from proof of Theorem \ref{mixdim}, we have $I(V_{\mathcal D})\supseteq I(V_{i_0})$, for some $i_0\in\{1,\ldots,m\}$. Consequently, $V_{\mathcal D}\subseteq V_{i_0}$, and hence $\mathcal D\subseteq \mathcal D_{i_0}$. We just proved the following result.

\begin{prop} With the notations and conditions of this section, we have that $\phi_{G_{\Lambda}}(V_i), i=1,\ldots m$ are the maximal subcodes of $\mathcal C_{\Lambda}$ of support size $\leq a-1$.
\end{prop}

Suppose $V$ is a union of points in $\mathbb P^{k-1}$; in other words, $V_i, i=1,\ldots,m$ are one-dimensional linear subspaces of $\mathbb K^k$. Then $V_{\mathcal D}$ from above must equal $V_{i_0}$. This leads to the conclusion that $\mathcal D_i, i=1,\ldots,m$ are the equivalence classes (under nonzero scalar multiplication) of minimal codewords of weight $\leq a-1$; see for comparison \cite[Proposition 4.1]{To1}.

\medskip

Still under the assumption that $V$ is a set of points, if in the proof of Theorem \ref{mixdim} we pick the $\Lambda_i$'s such that for all $i\neq j$, if $\ell\in \Lambda_j\setminus\Lambda_i$, then $\ell\notin I(V_i)$, then as observed before, $m(\mathcal D_i)=a-1$, and therefore $\mathcal D_i=\phi_{G_{\Lambda}}(V_i),i=1,\ldots,m$ are precisely \underline{all} the projective codewords of minimum weight of $\mathcal C_{\Lambda}$; the minimum weight (or distance) is $a-1$.

\subsection{Interpolating points in $\mathbb P^2$.} In this subsection we focus our attention towards interpolating points in the (projective) plane. Let $X:=\{P_1,\ldots,P_m\}\subset\mathbb P^2$, be a subset of $m$ distinct points, not all collinear.

If we use Theorem \ref{mixdim}, a priori we would pick $\aleph:=m+1$ generic (projective) lines passing through $P_1$, then another set of $\aleph$ generic lines through $P_2$, and so forth. So $n=m(m+1)$ and $a=n-\aleph+1=m^2$. This interpolation does the trick, yet it is not that effective in terms of $n$, which is the size of $\Lambda$. This idea doesn't use the geometry of $X$, in any way. For example, the linear form defining the line connecting $P_1$ and $P_2$ should belong to both $\Lambda_1$ and $\Lambda_2$. And of course, if $P_1$, $P_2$, and $P_3$ are collinear, this line should belong to $\Lambda_3$ as well. Also, when defining GSCV's we can allow some of the linear forms in $\Lambda$ to be proportional, whereas the proof of Theorem \ref{mixdim} doesn't seem to work with this loosen condition. Still, this naive interpolation gives the clue to what one should look for to obtain a better interpolation: the $n=m(m+1)$ linear forms of $\Lambda$ define a line arrangement in $\mathbb P^2$ whose singularities of maximum multiplicity (equal to $\aleph$) are precisely the points of $X$. Then, from \cite[Lemma 2.2]{To1}, $\sqrt{I_{n-\aleph+1}(\Lambda)}=I(X)$.

With the above discussions, the goal is to find a (better) line arrangement in $\mathbb P^2$ whose singularities of maximum multiplicity are exactly the points of $X$. Then apply \cite[Lemma 2.2]{To1} to interpolate $X$.

Let us denote $\mathcal A_X$ to be the multi arrangement in $\mathbb P^2$ constructed in the following way. We pick an ordering of the points of $X=\{P_1,\ldots,P_m\}$. For $1\leq i<j\leq m$ consider the line $\ell_{i,j}$ connecting the points $P_i$ and $P_j$. If such a line has $s$ points of $X$ on it, we consider it $s-1$ times. Let $\Lambda_X$ be the collection of all of the linear forms (considered with repetitions) defining these lines.

$\bullet$ Let $\mathcal X$ be the line arrangement in $\mathbb P^2$ with lines dual to the points of $X$, i.e., if $[a,b,c]\in X$, then $V(ax+by+cz)\in \mathcal X$. Then, under this duality, collinear points of $X$ become concurrent lines of $\mathcal X$. Then the size of $\Lambda_X$ is the sum of the M\"{o}bius function values at the rank 2 elements in the intersection lattice of $\mathcal X$ (also known as the intersection points of $\mathcal X$). In other words
\[
\#(\Lambda_X)=\frac{\pi_{\mathcal X}(1)}{2}-1,
\] where $\pi_{\mathcal X}(t)$, denotes the Poincar\'{e} polynomial of $\mathcal X$.

\medskip

$\bullet$ Denote $Sing(\mathcal A_X)$ the set of the intersection points of the lines of $\mathcal A_X$, and if $Q\in Sing(\mathcal A_X)$, denote with $\nu(Q,\mathcal A_X)$ the number of lines of $\mathcal A_X$, counted with multiplicity, that pass through $Q$. Obviously, $X\subseteq Sing(\mathcal A_X)$.

Under the same duality at the previous bullet, intersection points lying on a line of $\mathcal X$ correspond to lines of $\mathcal A_X$ passing through the corresponding point of $X$. But the sum of the M\"{o}bius function values at all these points on a line of $\mathcal X$ is exactly $|\mathcal X|-1=m-1$. In other words $$\nu(P_i,\mathcal A_X)=m-1,i=1,\ldots,m.$$

\medskip

$\bullet$ Next we show that if $Q\in Sing(\mathcal A_X)\setminus X$, then $\nu(Q,\mathcal A_X)< m-1$. Let $Q$ be such a point. Then $Q=\ell_1\cap\cdots\cap\ell_u$, where $\ell_1,\ldots,\ell_u$ are distinct lines of $\mathcal A_X$, hence $u\geq 2$. Suppose for each $i=1,\ldots,u$, the line $\ell_i$ has $n_i+1, n_i\geq 1$ points of $X$ on it. Since $Q\notin X$, these $u$ subsets of points of $X$ are all disjoint. Therefore
\[
\nu(Q,\mathcal A_X)=n_1+\cdots+n_u\leq |X|-u=m-u\leq m-2.
\]

\medskip

Now putting everything together, we have $$\mathcal V_{\#(\Lambda_X)-(m-1)+1}(\Lambda_X)=X.$$ Obviously $\displaystyle \#(\Lambda_X)\leq{{m}\choose{2}}$, so the interpolating GSCV has support size at most half of the support size of the GSCV considered in the proof of Theorem \ref{mixdim}.

\begin{ex}\label{example2} Let $$X=\{[0,0,1],[0,1,1],[0,2,1],[1,0,1],[1,1,1]\}.$$ Then $$\mathcal X=\{V(z),V(y+z), V(2y+z), V(x+z),V(x+y+z)\}.$$ From \cite[Lemmas 3.1 and 3.2]{Sc}, the primary decomposition of $I_{5-1}(\Sigma),$ where $\Sigma=(z,y+z,2y+z,x+z,x+y+z)$, will give all the intersection points of $\mathcal X$ and their M\"{o}bius function values, and therefore, dually we will obtain all the linear forms in $\Lambda_X$, and how many times they occur. The following calculations have been performed with Macaulay2 (\cite{GrSt}).

$$\langle y,z\rangle^2\cap\langle y,x+z\rangle \cap\langle z,x+y\rangle \cap\langle z,x\rangle \cap\langle y+z,x+z\rangle \cap\langle y+z,x\rangle \cap\langle 2y+z,x+z\rangle \cap\langle 2y+z,2x+z\rangle.$$

The first is the ideal of the point $[1,0,0]$, and it has multiplicity 2. So the linear form $x$ shows up twice in $\Lambda_X$. Doing this for all the other ideals we obtain

$$\Lambda_X=(x,x,x-z,x-y,y,x+y-z,y-z,2x+y-2z,x+y-2z).$$ We have indeed $$\sqrt{I_6(\Lambda_X)}=\langle x,y\rangle\cap \langle x,y-z\rangle \cap \langle x,y-2z\rangle\cap\langle y, x-z\rangle\cap\langle x-z,y-z\rangle.$$ The points of $X$ are the intersection points in the picture below where 4 lines intersect (the double line is counted as two lines).
\begin{center}
\epsfig{file=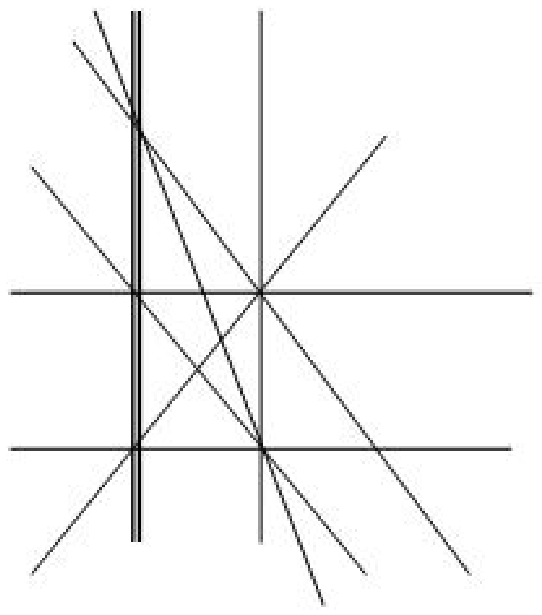,height=2.5in,width=2in}
\end{center}
\end{ex}

\section{Arithmetic rank of GSCS's}

Let $R$ be a commutative unitary Noetherian ring and let $I$ be a proper ideal of $R$. Suppose ${\rm ht}(I)=m$. Then $I$ is said to be {\em a set-theoretic complete intersection (s.t.c.i)} if there exist $f_1,\ldots,f_m\in I$ such that $\sqrt{I}=\sqrt{\langle f_1,\ldots,f_m\rangle}$. A variety is called {\em set-theoretic complete intersection} if its defining ideal has this property. The {\em arithmetic rank} of an ideal $I$, denoted $ara(I)$, is the minimum number of elements in $I$ that generate $I$ up to its radical ideal. So $I$ is a set-theoretic complete intersection if and only if $ara(I)={\rm ht}(I)$ (in general one has $\geq$ happening).

\medskip

\noindent {\bf Hartshorne's Example.} In the nineteenth century, Kronecker and Cayley conjectured that any complex variety in $\mathbb P^3$ is set-theoretic complete intersection. In \cite{Ha}, Hartshorne gives the following counterexample: consider the variety with defining ideal $I:=\langle x_1,x_2\rangle\cap\langle x_3,x_4\rangle$. Because the third local cohomology module of $\mathbb C[x_1,\ldots,x_4]$ supported at $I$ is not zero, the arithmetic rank is $ara(I)=3$, strictly greater than the height of $I$ which is ${\rm ht}(I)=2$. So $I$ is not a set-theoretic complete intersection.

\medskip

\begin{rem}\label{Hart} Hartshorne's Example is a projective subspace arrangement, therefore it can be interpolated by a GSCV. Applying the proof of Theorem \ref{mixdim}, we consider $\Lambda=(x_1,x_2,x_1+x_2,x_3,x_4,x_3+x_4)\subset \mathbb C[x_1,\ldots,x_4]$. Then $$\sqrt{I_4(\Lambda)}=\langle x_1,x_2\rangle\cap\langle x_3,x_4\rangle.$$ This calculation answers negatively a question asked in \cite{To2}, whether or not \underline{any} GSCS is set-theoretic complete intersection.
\end{rem}

\medskip

The next result presents an upper-bound on the arithmetic rank of any GSCS (or GSCV).

\begin{thm}\label{ara} Let $\Lambda=(\ell_1,\ldots,\ell_n)$, where $\ell_i\in R:=\mathbb K[x_1,\ldots,x_k]$ are linear forms generating the maximal ideal $\frak m$. Let $a\in\{1,\ldots,n\}$. Then $$ara(I_a(\Lambda))\leq n-a+1.$$
\end{thm}
\begin{proof} In \cite[Lemma on p. 249]{ScVo2} the following Lemma is proved: let $R$ be a commutative ring with non-zero identity. Let $P$ be a finite subset of elements of $R$. Let $P_0,\ldots, P_r$ be subsets of $P$ such that
\begin{enumerate}
  \item[(i)] $\displaystyle\bigcup_{l=0}^r P_l=P$;
  \item[(ii)] $P_0$ has exactly one element;
  \item[(iii)] if $p$ and $p''$ are different elements of $P_l (0<l\leq r)$, there is an integer $l'$ with $0\leq l'<l$, and an element $p'\in P_{l'}$ such that $p'|p\cdot p''$.
\end{enumerate} Setting $\displaystyle q_l=\sum_{p\in P_l}p^{e(p)},$ where $e(p)\geq 1$ are arbitrary integers, then $$\sqrt{\langle P\rangle}=\sqrt{\langle q_0,\ldots,q_r\rangle}.$$

Fix $j\in \{1,\ldots,n\}$. We apply the Lemma for the case when $P$ is the set of generators $\{\ell_I|I\subset \{1,\ldots,n\}, |I|=n-j\}$ for the ideal $I_{n-j}(\Lambda)$. We use the standard notation $\displaystyle \ell_I=\prod_{i\in I}\ell_i$.

Consider $$P_0:=\{\ell_{j+1}\cdots\ell_n\},$$ and for $u=1,\ldots,j$, $$P_u:=\{\ell_{j-u+1}\ell_I|I\subset \{j-u+2,\ldots,n\},|I|=n-j-1\}.$$

With $r=j$, properties (i) and (ii) in Lemma above are immediately satisfied.

For property (iii), let $u>0$ and let $p=\ell_{j-u+1}\ell_I$ and $p''=\ell_{j-u+1}\ell_{I''}$, with $I,I''\subset\{j-u+2,\ldots,n\}, |I|=|I''|=n-j-1,$ and $I\neq I''$ (so $p, p''\in P_u$, with $p\neq p''$).

Let $\alpha$ be the smallest element of $I$ and let $\beta$ be the smallest element of $I''$, so $\alpha,\beta\geq j-u+2$. Let $\gamma=\min\{\alpha,\beta\}\geq j-u+2$. Taking $u'=j-\gamma+1$, we have $j-u'+1=\gamma\geq j-u+2$ giving $u\geq u'+1$, hence $u'<u$.

Since $I\neq I''$, then $|I\cup I''|\geq n-j$, and therefore $|I\cup I''|\setminus\{\gamma\}|\geq n-j-1$. Therefore there exists $I'\subseteq I\cup I''\setminus\{\gamma\}$, with $|I'|=n-j-1$. Also $$\min\{i|i\in I'\subseteq I\cup I''\setminus\{\gamma\}\}\geq \gamma+1=j-u'+2,$$ hence $I'\subset\{j-u'+2,\ldots,n\}$.

Everything put together gives that $$p'=\ell_{j-u'+1}\ell_{I'}\in P_{u'}, u'<u,$$ and since $\{\gamma\}\cup I'\subset I\cup I''$, we have $p'=\ell_\gamma\ell_{I'}|\ell_{I\cup I''}|p\cdot p''$. This means that condition (iii) is also satisfied.

Taking $a=n-j$, we proved our result.
\end{proof}

\begin{rem} \label{stci} In \cite{To2}, by somewhat different methods, it is proved that any usual star configuration is s.t.c.i. We mention again that in order to have an usual star configuration, then any $k$ of the linear forms of $\Lambda$ are linearly independent. Let $1\leq c\leq k-1$, and suppose any $c$ of the linear forms of $\Lambda$ are linearly independent. Let $a=n-c+1$. Then any $n-a+1=c$ linear forms of $\Lambda$ are linearly independent which gives that ${\rm ht}(I_a(\Lambda))=n-a+1$. From Theorem \ref{ara}, one obtains that $I_a(\Lambda)$ is s.t.c.i. This way we answer the second question asked at the end of \cite{To2}.
\end{rem}

\subsection{Comments on the arithmetic rank of subspace arrangements.} After his example came to surface, Hartshorne conjectured that every irreducible curve in $\mathbb P^3$ is set-theoretic complete intersection. This conjecture is still unsolved. In regard to reducible varieties, Hartshorne asked some interesting questions derived from his example about the number of equations needed to define (up to radical) reducible varieties (\cite{Ha2}); they were answered in \cite{ScVo1} and \cite{Ly} (the base field is algebraically closed). A very good survey of this very beautiful topic of set-theoretic intersections is \cite{Ly2}.

Combining Theorem \ref{ara} and the proof of Theorem \ref{mixdim} (we have $n-a+1=\aleph$), we immediately have the following result, less known in the literature at least under the form we are presenting it here.

\begin{cor}\label{main} Let $V=V_1\cup\cdots\cup V_m$ be an essential (projective) subspace arrangement whose $m$ irreducible components have codimensions $c_i, i=1,\ldots,m$. Then $$ara(I(V))\leq 1+\sum_{i=1}^m(c_i-1).$$
\end{cor}

The proof of this corollary can be obtained in a different fashion, by applying the same proof of Theorem 5.1 (page 401) in \cite{KiTeYo}, by replacing the independent variables with the linear forms defining each irreducible component $V_i$. This Theorem 5.1 can give an entire class of examples when the bound in Corollary \ref{main} is attained, and generalizes the following example from \cite[Section 3]{ScVo1}: let $V$ be the subspace arrangement of $\mathbb P^{rt-1}$ over an algebraically closed field, with defining ideal $$I(V)=\langle x_1,\ldots,x_t\rangle\cap\langle x_{t+1},\ldots,x_{2t}\rangle\cap\cdots\cap\langle x_{(r-1)t+1},\ldots,x_{rt}\rangle.$$ It is proven that $ara(I(V))=r(t-1)+1$. From our perspective, $m=r$, and $c_i=t$ for all $i=1,\ldots,m$.

\medskip

\begin{rem} \label{rem1} Suppose $k-1\geq c_1\geq\cdots\geq c_m$. After a change of coordinates one can assume that $I(V_1)=\langle x_1,\ldots,x_{c_1}\rangle$. So $I(V)\subset \langle x_1,\ldots,x_{k-1}\rangle R$. Therefore, from \cite[Theorem 2]{EiEv}, we have that $ara(I(V))\leq k-1$. So the upper bound becomes non-trivial if $k$ is very big compared to the number of components and their codimension.
\end{rem}

\begin{ex} Let us go back to Example \ref{example}, for some in-depth analysis.  By Corollary \ref{main}, we have $ara(I)\leq 1+(2-1)+(3-1)=4$, whereas from Remark \ref{rem1} $ara(I)\leq 3$.

The question is if it is possible to interpolate $V$ with $\mathcal V_a(\Lambda)$, such that $n-a+1=3$, where $\Lambda$ consists of $n$ linear forms in $R:=\mathbb C[x,y,z,w]$. If that were the case, then $$\langle x,z,w\rangle=\langle \ell_{i_1},\ell_{i_2},\ell_{i_3}\rangle \mbox{ and }\langle x,y\rangle=\langle \ell_{j_1},\ell_{j_2},\ell_{j_3}\rangle,$$ where $\ell_{i_u},\ell_{j_v}\in\Lambda$, with $i_u$ not necessarily distinct than $j_v$, for some $u,v$'s.

We have $\ell_{i_u}=a_ux+c_uz+d_uw,u=1,2,3, a_u,c_u,d_u\in\mathbb K$, and the determinant of the $3\times 3$ matrix of coefficients of these three linear forms is not equal to 0.

Also $\ell_{j_v}=a'_vx+b_vy,v=1,2,3, a'_v,b_v\in\mathbb K$, are not all proportional. Let $J$ be the linear prime ideal generated by two of the $\ell_{i_u}$'s, both not proportional to $x$, and by one of the $\ell_{j_v}$'s, also not proportional to $x$. $V(J)$ is an irreducible component of $\mathcal V_a(\Lambda)$, but not of $V$. So the answer to our question is NO.
\end{ex}

\renewcommand{\baselinestretch}{1.0}
\small\normalsize 

\bibliographystyle{amsalpha}

\end{document}